\documentclass[letterpaper, 10 pt, conference]{ieeeconf} 
\IEEEoverridecommandlockouts  
\usepackage{graphics,graphicx} 
\usepackage[dvipsnames]{xcolor}
\usepackage{epsfig} 
\usepackage{amsmath} 
\usepackage{amssymb}  
\usepackage{amsfonts}
\usepackage{amsthm}
\usepackage{mathtools}  
\usepackage{optidef,booktabs}
\usepackage{cite}
\usepackage{url}
\usepackage{lipsum}
\usepackage{algorithm}
\usepackage{algpseudocode}
\usepackage[utf8]{inputenc} 
\usepackage{subcaption}
\usepackage{tabularray}
\usepackage{tikz}
\usetikzlibrary{arrows.meta, positioning, calc}
\usepackage[american,cuteinductors,smartlabels]{circuitikz}
\let\labelindent\relax
\usepackage[shortlabels]{enumitem}
\usepackage{soul}
\usepackage{hyperref} 
\hypersetup{ 		
	colorlinks=true,        
	linkcolor=blue,
	urlcolor=blue,
	citecolor=blue    		
}

\expandafter\def\expandafter\normalsize\expandafter{%
    \normalsize%
    \setlength\abovedisplayskip{1pt}%
    \setlength\belowdisplayskip{1pt}%
    \setlength\abovedisplayshortskip{1pt}%
    \setlength\belowdisplayshortskip{1pt}%
}
\UseTblrLibrary{amsmath}

\newtheorem{proposition}{Proposition}
\theoremstyle{remark}
\newtheorem{remark}{Remark}

\theoremstyle{definition}
\newtheorem{lemma}{Lemma}
\newtheorem{theorem}{Theorem}
\newtheorem{corollary}{Corollary}
\theoremstyle{definition}
\newtheorem{definition}{Definition}

\DeclareMathOperator*{\dist}{d}

\providecommand{\norm}[1]{\lVert#1\rVert}

\providecommand{\norm}[1]{\lVert#1\rVert}

\newcommand{\R}{\mathbb{R}}

\newcommand{\Ltwo}{\mathcal{L}_2}
\newcommand{\Ltwoe}{\mathcal{L}_{2e}}
\newcommand{\inner}[2]{\langle #1, #2 \rangle}
\newcommand{\SG}{\mathrm{SG}}
\newcommand{\SGe}{{\mathrm{SG}_e}}

\begin{document}

\title{Scaled Graph Containment for Feedback Stability:  Soft–Hard Equivalence and Conic Regions}

\author{Eder~Baron-Prada, Julius P. J. Krebbekx, Adolfo Anta and~Florian Dörfler
 \thanks{Eder Baron is with the Austrian Institute of Technology, 1210 Vienna, Austria, and also with the Automatic Control Laboratory, ETH Zurich, 8092 Z\"urich, Switzerland. (e-mail: ebaron@ethz.ch)}
 \thanks{Julius P. J. Krebbekx is with the Control Systems group, Department of Electrical Engineering, Eindhoven University of Technology, The Netherlands (email: j.p.j.krebbekx@tue.nl).}
 \thanks{Adolfo Anta is with the Austrian Institute of Technology, Vienna 1210, Austria (e-mail: adolfo.anta@ait.ac.at).}
 \thanks{Florian Dorfler is with the Automatic Control Laboratory, ETH Zürich, Zürich 8092, Switzerland (e-mail: dorfler@ethz.ch).}
}

\maketitle

\begin{abstract}
Scaled graphs (SGs) offer a geometric framework for feedback stability analysis. This paper develops containment conditions for SGs within multiplier-defined regions, addressing both circular and conic geometries. For circular regions, we show that soft and hard SG containment are equivalent whenever the associated multiplier is positive-negative. This enables hard stability certification from soft computations alone, bypassing both the positive semidefinite storage constraint and the homotopy condition of existing methods. Numerical experiments on systems with up to 300 states demonstrate computational savings of 15–44\% for the circular containment framework. We further characterize which conic regions are hyperbolically convex, a condition our frequency-domain certificate requires, and demonstrate that such regions provide tighter SG bounds than circles whenever the operator SG is nonsymmetric.
\end{abstract}


\IEEEpeerreviewmaketitle
\section{Introduction}
\label{sec:intro}

Graphical methods have been central to feedback stability analysis since the work of Nyquist and Bode, offering intuitive geometric certificates.  Extending this paradigm to MIMO LTI and nonlinear systems has been a long-standing goal. Scaled graphs (SGs), introduced  in~\cite{ryu2022large} and developed for feedback analysis in~\cite{Chaffey_2023}, address this gap by representing operators as subsets of the complex plane that simultaneously encode gain and phase.

Several SG formulations have been proposed, including signed~\cite{Sebastian2025_signed}, soft~\cite{Chaffey_2023}, and hard~\cite{chen2025softhardSRG} variants. This study focuses on the latter two. The \emph{soft SG} is defined on $\mathcal{L}_2$ and characterizes the asymptotic behavior of $\mathcal{L}_2$-stable operators; the \emph{hard SG} is defined on $\mathcal{L}_{2e}$, capturing finite-horizon behavior and accommodating persistent and unbounded trajectories.

Two fundamental limitations currently constrain SG-based stability analysis. The first is \emph{geometric}; existing SG containment methods~\cite{degroot2025dissipativity,nauta2025computable,krebbekx2025computing} rely on static multiplier regions, which describe only circular regions (disks, disk complements, or half-planes). Some specifications of practical interest, such as elliptical exclusion zones  require \emph{conic} containment regions that lie outside this class\cite{Gupta_2023_elipsoidal}. The second limitation is \emph{computational}; soft SGs are efficiently constructed via LMIs~\cite{degroot2025dissipativity} or frequency-domain sampling~\cite{Baron2025SRG,krebbekxGraphicalAnalysisNonlinear2025}, but certifying stability from them requires checking separation conditions over infinitely many scaled sets~\cite{Chaffey_2023}, a condition that is numerically expensive to verify. Hard SGs yield a simple geometric separation test~\cite{chen2025softhardSRG}, but their computation demands a semidefinite constraint $P \succeq 0$ that becomes costly for large-scale systems~\cite{krebbekx2025computing,nauta2025computable,degroot2025dissipativity}.

This paper addresses both limitations. For the geometric limitation, we extend the containment framework from circular regions to general conic sections, following the frequency-domain inequality framework of~\cite{IwasakiHara2005}. We characterize exactly which conic regions are hyperbolically convex, the geometric property required for SG containment, and provide a frequency-domain certification condition for conic containment. For the computational limitation, we use the lens of Integral Quadratic Constraints  (IQCs)~\cite{Megretski1997,Carrasco_2015_IQCsseparation}. We show that when the circular IQC multiplier is  \emph{positive-negative}~\cite{Carrasco_2018,Seiler2015}, soft and hard SG containment in a multiplier-defined region are equivalent for $\mathcal{L}_2$-stable systems. This enables a hybrid certification pipeline: compute SG regions via efficient soft LMIs, then certify stability via the simple hard separation condition, bypassing both the $P \succeq 0$ constraint and the homotopy sweep. 

The practical implications are significant. Numerical experiments on LTI systems with up to 300 states demonstrate computational savings of 15--44\% from the soft LMI relaxation alone, enabling SG-based stability certification for large-scale systems, including power systems~\cite{baronprada2026powersystems} and networked multi-agent systems~\cite{Baron2025decentralized}, where direct hard SG computation is currently prohibitive. Beyond computational efficiency, when the SG is elongated, a conic region can be shaped to fit it more closely than any disk: since stability certificates are derived from the containment region, tighter regions directly translate into less conservative stability margins.

\section{Preliminaries}\label{sec:prelim}

\subsection{Signal Spaces, Systems, and Scaled Graphs}

Let $\Ltwo^n$ denote the space of square-integrable signals $u:\R_{\ge 0}\to\R^n$ with norm $\norm{u}^2 := \int_0^\infty u(t)^\top u(t)\,dt$ and inner product $\inner{u}{y} := \int_0^\infty u(t)^\top y(t)\,dt$.  The extended space $\Ltwoe^n := \{ u:\R_{\ge 0}\to\R^n \mid \mathcal{P}_T u \in \Ltwo^n,\, \forall T \ge 0 \}$ accommodates signals with unbounded energy, where $\mathcal{P}_T$ truncates to $[0,T]$.  A system $H:\Ltwoe^n \to \Ltwoe^n$ is \emph{causal} if $\mathcal{P}_T(Hu) = \mathcal{P}_T(H(\mathcal{P}_T u))$ for all $T\ge 0$~\cite{zhou1998}.

We also consider finite-dimensional LTI systems with zero initial condition, i.e., $x(0)=0$, and $\dot{x}(t) = Ax(t) + Bu(t), \;y(t) = Cx(t) + Du(t)$, where $x \in \mathbb{R}^m$ is the state vector, $u \in \mathbb{R}^n$ is the input, and $y \in \mathbb{R}^n$ is the output, with appropriately dimensioned matrices $A$, $B$, $C$, and $D$. The system has transfer function $H(s) = C(sI-A)^{-1}B + D$, where $x \in \mathbb{R}^m$, $u,y \in \mathbb{R}^n$.  The system is \emph{$\Ltwo$-stable} if $\norm{H}_\infty := \sup_\omega \bar\sigma(H(\textup{j}\omega)) < \infty$.

SGs generalize input--output pairs to gain--phase points in~$\mathbb{C}\cup {\infty}$~\cite{Chaffey_2023}.  For $u,y\in\Ltwo^n$ with $u\neq 0$, define

\begin{align} \label{eqn:gain_phase}
    \rho(u,y) := \tfrac{\norm{y}}{\norm{u}}, \quad \theta(u,y) := \arccos\!\left(\tfrac{\inner{u}{y}}{\norm{u}\norm{y}}\right),
\end{align}
with $\theta(u,y)=0$ if $y=0$.  The \emph{soft SG} of $H:\Ltwo\to\Ltwo$ is

\begin{align*}
\SG(H):= \left\{ \rho(u,y)\,e^{\pm\mathrm{j}\theta(u,y)}\mid \forall u\in\Ltwo\setminus\{0\},\, y= Hu \right\},
\end{align*}
and the \emph{hard SG} of $H:\Ltwoe\to\Ltwoe$, using truncated gain $\rho_T := \rho(\mathcal{P}_T u,\mathcal{P}_T y)$ and phase $\theta_T := \theta(\mathcal{P}_T u,\mathcal{P}_T y)$, is

\begin{align*}
    \SGe(H)\!:=\!\left\{\rho_T(u,y)e^{\pm\mathrm{j}\theta_T(u,y)}\!\mid\! \forall u\!\in\Ltwoe, y= Hu, \forall T\!>\!0 \!\right\}.
\end{align*}

Let $\mathbb{S}^n$ denote the real symmetric $n \times n$ matrices. The open upper half-plane is $\mathbb{C}_+ := \{z \in \mathbb{C} : \operatorname{Im}\{z\} > 0\}$. A set $\mathcal{R} \subset \mathbb{C}_+$ is \emph{hyperbolically convex} (h-convex) if for every pair of points in $\mathcal{R}$, the geodesic connecting them lies in $\mathcal{R}$ (See \cite{pates2021scaled} for details).

\subsection{IQCs and $J$-Spectral Factorization}
\label{subsec:iqc}

A system $H:\Ltwo^n\to\Ltwo^n$ satisfies a \emph{soft IQC} defined by a Hermitian constant multiplier $\Pi\in\mathbb{S}^{2n}$ if
\begin{align}\label{eq:soft_iqc}
\int_0^\infty
\begin{bmatrix} y \\ u \end{bmatrix}^\top
\Pi 
\begin{bmatrix} y \\ u \end{bmatrix}
dt \ge 0, \, \forall\, u \in \Ltwo^n,\; y=Hu.
\end{align}

It satisfies a \emph{hard IQC} if for any $T \ge 0$
\begin{align}\label{eq:hard_iqc}
\int_0^T
\begin{bmatrix} y \\ u \end{bmatrix}^\top
\Pi 
\begin{bmatrix} y \\ u \end{bmatrix}
dt \ge 0, \, \forall\, u \in \Ltwoe^n,\; y=Hu.
\end{align}

The passage from soft to hard IQCs is governed by the factorization structure of~$\Pi$\cite{Megretski1997}.

\begin{definition}[$J$-Spectral Factorization\cite{Seiler2015}]
\label{def:j_spectral}
Let $\Pi = \Pi^\top \in \mathbb{S}^{2n}$ be Hermitian. 
We say that $\Pi$ admits a \emph{$J$-spectral factorization} if there exists an invertible matrix $\Psi \in \mathbb{R}^{2n\times 2n}$ such that $ \Pi = \Psi^\top J \Psi,$ where $J := \operatorname{diag}(I_n,-I_n)$.
\end{definition}


\begin{definition}[Positive-Negative Matrix]
\label{def:pos-neg}
Let $\Pi = \Pi^\top \in \mathbb{C}^{2n\times 2n}$ be partitioned as $\Pi = \begin{bmatrix} \Pi_{11} & \Pi_{12} \\\Pi_{21} & \Pi_{22}\end{bmatrix},$ where $\Pi_{ij} \in \mathbb{R}^{n\times n}$. 
We say that $\Pi$ is \emph{positive-negative} if there exists $\varepsilon > 0$ such that $\Pi_{22} \succeq \varepsilon I_n$ and $\Pi_{11} \preceq -\varepsilon I_n.$
\end{definition}

Every positive-negative multiplier admits a $J$-spectral factorization $\Pi = \Psi^\top J\, \Psi$~\cite{Seiler2015}; moreover, this factorization is doubly hard~\cite{Carrasco_2018}, which is required for the following soft--hard equivalence result based on~\cite{Seiler2015,Carrasco_2018}.

\begin{lemma}[Soft--hard IQC equivalence]
\label{lem:iqc_equiv}
Let $H:\Ltwo^n\to\Ltwo^n$ be a stable causal operator and let  $\Pi\in\mathbb{S}^{2n}$ be a constant positive-negative multiplier. Then $H$ satisfies~\eqref{eq:soft_iqc} if and only if $H$ satisfies~\eqref{eq:hard_iqc}.
\end{lemma}

\begin{proof}
$(\Rightarrow)$\; Since $\Pi$ is constant and positive-negative, it admits a $J$-spectral factorization $\Pi = \Psi^\top J\, \Psi$ where $\Psi\in\mathbb{R}^{2n\times 2n}$ is a constant invertible matrix~\cite[Lemma~4]{Seiler2015}.  Setting $z(t) := \Psi\Bigl[\begin{smallmatrix}(Hu)(t)\\u(t)\end{smallmatrix}\Bigr]$, the pointwise identity $\bigl[\begin{smallmatrix}Hu\\u\end{smallmatrix}\bigr]^\top \Pi\,\bigl[\begin{smallmatrix}Hu\\u\end{smallmatrix}\bigr] = z^\top J\, z$ yields
\begin{align}\label{eq:factor_identity}
\int_0^T \begin{bmatrix}Hu\\u\end{bmatrix}^\top \Pi\,\begin{bmatrix}Hu\\u\end{bmatrix}\,dt 
=\! \int_0^T z^\top J\, z\, dt, \quad \forall\, T \leq \infty.
\end{align}

The soft IQC~\eqref{eq:soft_iqc} is therefore equivalent to  $\int_0^\infty z^\top J\, z\, dt \geq 0$.  Since $\Psi$ is constant the factorization $(\Psi,J)$ has no internal states; applying \cite[Lemma~2]{Seiler2015} with the storage $\overline{M}=\underline{M}=0$ (which holds as $\Pi$ is positive-negative \cite[Thm~4.1]{Carrasco_2018} \cite[Thm~4]{Seiler2015}), the truncated inequality $\int_0^T z^\top J\,z\,dt \geq 0$ holds for all $T\geq 0$ and all $u\in\Ltwoe^n$.  It follows from~\eqref{eq:factor_identity} that~\eqref{eq:hard_iqc} holds.

$(\Leftarrow)$\; Since $H$ is $\Ltwo$-stable, for every $u\in\Ltwo^n$ both $u$ and $Hu$ lie in $\Ltwo^n$, so the integrand in~\eqref{eq:hard_iqc} is absolutely integrable and passing $T\to\infty$ yields~\eqref{eq:soft_iqc}.
\end{proof}

\section{The SG--IQC Connection}\label{sec:connection}

This section formalizes the relationship between IQCs and SGs. Whereas IQCs describe operators via quadratic inequalities on signals, SGs represent these same constraints geometrically as regions in the complex plane.  The development here builds upon the framework established in~\cite{degroot2025dissipativity}.

\subsection{From IQCs to regions in the complex plane}

Consider constant multipliers $\Pi \in \mathbb{S}^{2}$ of the form
\begin{align}\label{eq:Pi_scalar}
\Pi = \begin{bmatrix} a & b \\ b & c \end{bmatrix}, \quad a, b, c \in \mathbb{R}.
\end{align}
By normalizing the dissipation inequality by the input signal energy, the IQC can be interpreted as a geometric constraint on the gain--phase pairs of the operator encoded in the multiplier.  These pairs are precisely the elements represented by the SG.

\begin{definition}[Multiplier Region]\label{def:S_Pi} 
Given $\Pi$ of the form~\eqref{eq:Pi_scalar}, the associated region in the complex plane is defined by \cite{degroot2025dissipativity} 
\begin{align}\label{eq:S_Pi}
\mathcal{S}(\Pi) :=& \left\{ z \in \mathbb{C} \;\middle|\; \begin{bmatrix} z \\ 1 
\end{bmatrix}^* \Pi \begin{bmatrix} z \\ 1 \end{bmatrix} \geq 0 \right\}\\ 
=& \left\{ z \in \mathbb{C} \mid a|z|^2 + 2b \Re\{z\} + c \geq 0 
\right\},\nonumber
\end{align}
which represents a disk or the exterior of a  disk~\cite{degroot2025dissipativity}. The interior (closed) disc is $\mathbb{D}^{\mathrm{int}}(c,r) := \{ z \in \mathbb{C} \mid |z-c| \leq r \}$.
The exterior (open) disc is $\mathbb{D}^{\mathrm{ext}}(c,r) := \{ z \in \mathbb{C} \mid |z-c| > r \}$.

\end{definition}

\subsection{IQC--SG correspondence}

The correspondence between static IQCs and SGs was established in \cite{degroot2025dissipativity}. We recall the results for both soft and hard IQCs.

\begin{lemma}[Soft IQC--SG Correspondence {\cite[Lemma~4]{degroot2025dissipativity}}]\label{lem:soft_connection}
Let $H : \Ltwo^n \to \Ltwo^n$ be an $\Ltwo$-stable system and let $\Pi$ as in~\eqref{eq:Pi_scalar}. Then $H$ satisfies the soft IQC defined by $\Pi$, i.e.,
\begin{align*}
\int_0^\infty \begin{bmatrix} y(t) \\ u(t) \end{bmatrix}^\top (\Pi\otimes I_n) \begin{bmatrix} y(t) \\ u(t) \end{bmatrix} dt \geq 0, \, \forall u \in \Ltwo^n,\ y = H(u),
\end{align*}
if and only if $\SG(H) \subset \mathcal{S}(\Pi)$.
\end{lemma}

An analogous correspondence holds for hard IQCs, where truncated signals are considered. 

\begin{lemma}[Hard IQC--SG Correspondence {\cite[Sec~6]{degroot2025dissipativity}}]\label{lem:hard_connection}
Let $H : \Ltwoe^n \to \Ltwoe^n$ be a causal system and let $\Pi$ as in~\eqref{eq:Pi_scalar}. Then $H$ satisfies the hard IQC defined by $\Pi$, i.e., for $ y = H(u),$
\begin{align*} 
\int_0^T \begin{bmatrix} y(t) \\ u(t) \end{bmatrix}^\top (\Pi\otimes I_n)  \begin{bmatrix} y(t) \\ u(t) \end{bmatrix} dt \geq 0,  \forall T > 0,\ u \in \Ltwoe^n,\
\end{align*}
if and only if $\SGe(H) \subset \mathcal{S}(\Pi)$.
\end{lemma}


\section{Scaled Graph Containment}
\label{sec:containment}

This section develops the containment machinery for SGs.  We begin with circular regions, where the $J$-spectral factorization aligns the soft and hard certification pathways.  We then extend the framework to conic regions.

\subsection{Soft and Hard SG Region Containment}
\label{ssec:soft-hard}

Leveraging the SG--IQC correspondence established in Section~\ref{sec:connection}, the classical soft--hard IQC equivalence is recast as a geometric statement on SGs.

\begin{theorem}[Equivalence of soft and hard SG containment]
\label{thm:main}
Let $H : \Ltwoe^n \to \Ltwoe^n$ be a causal, $\Ltwo$-stable 
system, and let $\Pi$ be a positive-negative multiplier as 
in~\eqref{eq:Pi_scalar}. Then
\begin{align*}
  \SG(H) \subset \mathcal{S}(\Pi)   
  \;\Longleftrightarrow\;   
  \SGe(H) \subset \mathcal{S}(\Pi).
\end{align*}
\end{theorem}

\begin{proof}
The result follows from the chain of equivalences:
\begin{align*}
  \SG(H) \subset \mathcal{S}(\Pi) 
  &\xLeftrightarrow{\;\text{Lem.~\ref{lem:soft_connection}}\;} 
  \text{soft IQC~\eqref{eq:soft_iqc}} \\
  \text{soft IQC~\eqref{eq:soft_iqc}} 
  &\xLeftrightarrow{\;\text{Lem.~\ref{lem:iqc_equiv}}\;} 
  \text{hard IQC~\eqref{eq:hard_iqc}} \\
  \text{hard IQC~\eqref{eq:hard_iqc}}
  &\xLeftrightarrow{\;\text{Lem.~\ref{lem:hard_connection}}\;} 
  \SGe(H) \subset \mathcal{S}(\Pi).  \qedhere
\end{align*}
\end{proof}



\subsection{Implications for SG Computations using LMIs}

We now restrict the attention to LTI systems.

\begin{corollary}\label{cor:lmi}
Let $H$ be an $\Ltwo$-stable LTI with stabilizable and detectable realization $(A,B,C,D)$, and let $\Pi$ be positive-negative.  Let
\begin{align}
  \varrho(\Pi) = \begin{bmatrix} C & D \\ 0 & I \end{bmatrix}^{\!\top}  (\Pi \otimes I_n)  \begin{bmatrix} C & D \\ 0 & I \end{bmatrix},
\end{align}
and consider the KYP-type LMI
\begin{align}\label{eq:soft_lmi}
  \begin{bmatrix} A & B \\ I & 0 \end{bmatrix}^{\!\top}  \begin{bmatrix} 0 & P \\ P & 0 \end{bmatrix}  \begin{bmatrix} A & B \\ I & 0 \end{bmatrix}  - \varrho(\Pi) \preceq 0,
\end{align}
with $P = P^\top$.  Then $\SG(H) \subset \mathcal{S}(\Pi)$ and $\SGe(H) \subset \mathcal{S}(\Pi)$.
\end{corollary}

\begin{proof}
The LMI~\eqref{eq:soft_lmi} yields $\SG(H) \subset \mathcal{S}(\Pi)$ by~\cite[Theorem 9]{degroot2025dissipativity}.  Because $\Pi$ is positive-negative, Theorem~\ref{thm:main} directly implies $\SGe(H) \subset \mathcal{S}(\Pi)$.
\end{proof}

\begin{proposition}[Circular-Region Multipliers]\label{prop:circular}
Let $c \in \mathbb{R}$, $r>0$, and consider the multipliers 
\begin{align*}
  \Pi_{\mathrm{int}}(c,r) :=    \begin{bmatrix} -1 & c \\ c & r^2 - c^2 \end{bmatrix}, \,  \Pi_{\mathrm{ext}}(c,r) :=    \begin{bmatrix} 1 & -c \\ -c & c^2 - r^2 \end{bmatrix}.
\end{align*}
Then $\mathcal{S}(\Pi_{\mathrm{int}}) = {\mathbb{D}}^{\mathrm{int}}(c,r)$ and $\mathcal{S}(\Pi_{\mathrm{ext}}) = \overline{\mathbb{D}}^{\mathrm{ext}}(c,r)$~\cite{degroot2025dissipativity}.  Of the two, only $\Pi_{\mathrm{int}}$ is positive-negative, which holds if and only if $r > |c|$.
\end{proposition}

\begin{proof}
The regions follow 
from~\cite[Lemma~4]{degroot2025dissipativity}.  
For the positive-negative claim, 
$(\Pi_{\mathrm{int}})_{11} = -1 < 0$ and 
$(\Pi_{\mathrm{int}})_{22} = r^2 - c^2 > 0$ if and only if 
$r > |c|$.  For $\Pi_{\mathrm{ext}}$, 
$(\Pi_{\mathrm{ext}})_{11} = 1 > 0$, violating 
$\Pi_{11} \preceq -\varepsilon I$.
\end{proof}

\begin{remark}[Advantages of the Soft LMI Formulation]\label{remark:computational}
The soft SG calculation via Corollary~\ref{cor:lmi} requires solving~\eqref{eq:soft_lmi} with $P = P^\top$ unconstrained, whereas the direct hard SG calculation~\cite{degroot2025dissipativity} imposes the constraint $P \succeq 0$.  This extra semidefinite cone raises per-iteration cost, increases the effective barrier parameter, and forces smaller step sizes when $P$ is nearly singular~\cite{sturm1999using,nesterov1994interior}.  The soft formulation avoids these effects, yielding speedups of up to 44\% (See Section~\ref{sec:scalability}).
\end{remark}

\subsection{From Circular to Conic Containment Regions}
\label{ssec:obstructions}

The circular regions of Sections~\ref{ssec:soft-hard} cover important cases such as gain bounds, passivity, but cannot capture more general conic exclusion regions (circles, ellipses, paraboloids and hyperboloids)~\cite{IwasakiHara2005,Gupta_2023_elipsoidal}. 

The Hermitian structure of the multiplier $\Pi\in\mathbb{S}^{2}$ forces $(\operatorname{Re}\{z\})^{2}$ and $(\operatorname{Im}\{z\})^{2}$ to share a single coefficient, so every $\mathcal{S}(\Pi)$ with $\det(\Pi)<0$ is a disk, disk complement, or half-plane.  The minimal extension that decouples these coefficients is~$\mathbb{S}^{3}$.  For $z \in \mathbb{C}$, define the augmented vector
\begin{align}\label{eq:augmented-vector}
  v(z) :=   [\operatorname{Re}\{z\}\;\operatorname{Im}\{z\}\;1]^\top \in\mathbb{R}^{3},
\end{align}
and for $\Theta\in\mathbb{S}^{3}$ the \emph{conic region}
\begin{align}\label{eq:conic-region}
  \mathcal{C}(\Theta) := \bigl\{z\in\mathbb{C} \;\big|\; v(z)^{\top}\Theta\,v(z) \leq 0 \bigr\}.
\end{align}
The quadratic form weights $(\operatorname{Re}\{z\})^{2}$ and $(\operatorname{Im}\{z\})^{2}$ independently. This is precisely the setting in~\cite{IwasakiHara2005}, where $3\times 3$ real-symmetric matrices characterize frequency-domain inequalities involving both $H(\textup{j}\omega)$ and its conjugate. For $\mathcal{C}(\Theta)\subset\mathbb{C}$, $\Theta$ must be indefinite, otherwise $\mathcal{C}(\Theta)$ is either empty or all of~$\mathbb{C}$~\cite{IwasakiHara2005}.

Since general conics need not be h-convex, a characterization of which matrices~$\Theta$ yield h-convex regions is required for Theorem \ref{thm:conic-containment}.  Since SGs are symmetric about the real axis, we restrict to $\Theta\in\mathbb{S}^{3}$ with $\Theta_{12}=0$ and $\Theta_{23}=0$, i.e.,
\begin{align}\label{eq:Theta-aligned}
  \Theta  =  \begin{bmatrix} \Theta_{11} & 0 & \Theta_{13}\\     0 & \Theta_{22} & 0\\    \Theta_{13} & 0 & \Theta_{33}  \end{bmatrix}
  \in\mathbb{S}^{3}.
\end{align}

For~\eqref{eq:Theta-aligned}, the region~\eqref{eq:conic-region} restricted to~$\mathbb{C}_{+}$ is
\begin{align}\label{eq:conic-aligned}
  \mathcal{C}(\Theta) \!  = \!\bigl\{\!z \!\in\! \mathbb{C}_{+}\! : \!    \Theta_{11}x^{2}\!+\Theta_{22}y^{2}\!+2\Theta_{13}x\!+\Theta_{33}\!\leq\! 0\!\bigr\},
\end{align}
with $z = x+\textup{j}y$.  $\Theta_{11}$ and~$\Theta_{22}$ govern the weighting of real and imaginary parts. We define $  \alpha := \Theta_{11} - \Theta_{22}.$

\begin{theorem}\label{thm:h-convexity}
Let $\Theta\in\mathbb{S}^{3}$ be of the form~\eqref{eq:Theta-aligned} and assume~$\Theta$ is indefinite. Then
\begin{align*}
  \mathcal{C}(\Theta)\ \textup{is h-convex}   \quad\Longleftrightarrow\quad  \Theta_{11}\geq\Theta_{22}.
\end{align*}
\end{theorem}
Prior to proving Theorem~\ref{thm:h-convexity}, we introduce a Beltrami--Klein coordinate representation of conic regions. Specifically, via the BK mapping $f_{\mathrm{BK}}:\mathbb{C}_{+}\to\mathbb{D}$, from coordinates $z=(x,y)\in\mathbb{C}_{+}$ to coordinates $w=(\eta,\phi) \in \mathbb{D}$ given by $\eta = \tfrac{x^{2}+y^{2}-1}{1+x^{2}+y^{2}},\,  \phi  = \tfrac{-2x}{1+x^{2}+y^{2}},$ and inverse (for $y>0$)
\begin{align}\label{eq:BK-inv}
  x = \tfrac{-\phi }{1-\eta}\,,\quad
  y = \tfrac{\sqrt{1-\eta^{2}-\phi ^{2}}}{1-\eta}\,,
\end{align}
where $1-\eta>0$ and $\sqrt{1-\eta^{2}-\phi ^{2}}>0$ on~$\mathbb{D}$. substituting~\eqref{eq:BK-inv} into~\eqref{eq:conic-aligned} and factorizing the positive factor $(1-\eta)^{2}$ yields the BK conic condition $q(w)\leq 0$ with
\begin{align}\label{eq:conic_BK}
       q(w):= w^{\!\top}\!M\,w + 2\,b^{\!\top}w + c,
\end{align}
where
\begin{align}\label{eq:Mbc-def}
\begin{split}
   M &= \begin{pmatrix} 
    \Theta_{33}\!-\!\Theta_{22} & \Theta_{13}\\
    \Theta_{13} & \Theta_{11}\!-\!\Theta_{22}
  \end{pmatrix},\\
  b& = \begin{pmatrix}-\Theta_{33}\\-\Theta_{13}\end{pmatrix}\!,\, c = \Theta_{22}+\Theta_{33}.
\end{split}
\end{align}

The conic in BK coordinates is $\mathcal{C}_{\mathrm{BK}}(\Theta) := \{w\in\mathbb{D}: q(w)\leq 0\}$, and h-convexity of $\mathcal{C}(\Theta)$ is equivalent to Euclidean convexity of~$\mathcal{C}_{\mathrm{BK}}(\Theta)$.

\begin{proof}[Proof of Theorem~\ref{thm:h-convexity}]
Write $\alpha:=\Theta_{11}-\Theta_{22}$.

\textbf{Degenerate case ($\Theta_{22}=0$).}
The conic~\eqref{eq:conic-aligned} reduces to $\Theta_{11}x^{2}+2\Theta_{13}x+\Theta_{33}\leq 0$, a condition on~$x$ alone.  If $\alpha\geq 0$, the set is a vertical strip or half-plane in~$\mathbb{C}_{+}$, which is h-convex.  If $\alpha<0$, the set is the exterior of a bounded interval, hence disconnected and not h-convex.

\textbf{Generic case ($\Theta_{22}\neq 0$).}
Since $\mathbb{D}$ is convex, the set $\mathcal{C}_{\mathrm{BK}}(\Theta)=\{q\leq 0\}\cap\mathbb{D}$ is convex if and only if its boundary arc $\Gamma:=\{w\in\mathbb{D}:q(w)=0\}$ has non-negative curvature. The signed curvature of~$\Gamma$ is $\kappa = N/\lvert\nabla q\rvert^{3}$, where
\begin{align}\label{eq:kappa-def}
  N := q_{\eta\eta}\,q_{\phi}^{2}-2\,q_{\eta\phi}\,q_{\eta}\,q_{\phi}+q_{\phi\phi}\,q_{\eta}^{2}
\end{align}
is the curvature numerator~\cite{GOLDMAN2005632} and subscripts denote partial derivatives.  Indefiniteness of~$\Theta$ ensures that~$\Gamma$ is a 
non-degenerate conic arc and therefore $\lvert\nabla q\rvert>0$ on~$\Gamma$, meaning $\operatorname{sign}(\kappa)=\operatorname{sign}(N)$ on~$\Gamma$ \cite{do2016differential}.

We now show that $N$ is constant on~$\Gamma$ and determined entirely by~$\alpha$.  Set $g := \tfrac{1}{2}\nabla q= Mw + b $.  The constant second derivatives $q_{\eta\eta}=2m_{11}$, $q_{\phi \phi }=2m_{22}$, $q_{\eta \phi }=2m_{12}$ together with $(q_\eta ,q_\phi ) = 2g^{\!\top}$ reduce~\eqref{eq:kappa-def} to
\begin{align}\label{eq:N-8gAg}
  N = 8\,g^{\!\top}\operatorname{adj}(M)\,g,
\end{align}
where $\operatorname{adj}(M) = \bigl(\begin{smallmatrix}m_{22}&-m_{12}\\-m_{12}&m_{11}\end{smallmatrix}\bigr)$.  Expanding $g = Mw+b$ and using $M^{\!\top}\operatorname{adj}(M)=\det(M)\,I_{2}$ yields
\begin{align}\label{eq:gAg-identity}
  g^{\!\top}\!\operatorname{adj}(M)g\!  =\! \det(M)\!\bigl[w^{\!\top}\!M\!w\! +\! 2b^{\!\top}\!w\bigr]+ b^{\!\top}\!\operatorname{adj}(M)b.
\end{align}
 Evaluating along the zero set, i.e., on $\Gamma$ the bracketed term in \eqref{eq:gAg-identity} equals $-c$, so~\eqref{eq:gAg-identity} becomes $b^{\!\top}\operatorname{adj}(M)\,b - c\,\det M$.  Substituting the entries from~\eqref{eq:Mbc-def}, every monomial in~$\Theta_{13}$ or~$\Theta_{33}$ cancels pairwise, leaving
\begin{align}\label{eq:key-cancel}
  b^{\!\top}\operatorname{adj}(M)\,b - c\,\det M
  = \Theta_{22}^{2}\,\alpha.
\end{align}
Combining~\eqref{eq:N-8gAg}--\eqref{eq:key-cancel} gives $N = 8\,\Theta_{22}^{2}\,\alpha$ on~$\Gamma$, independently of~$\Theta_{13}$, $\Theta_{33}$, or the position~$w$.

If $\alpha\geq 0$, then $N\geq 0$ and $\mathcal{C}_{\mathrm{BK}}(\Theta)$ is convex.  If $\alpha<0$, then $N<0$, the boundary is locally strictly concave as seen from~$\{q\leq 0\}$, and a chord between two nearby boundary points on opposite sides of any $w_{0}\in\Gamma\cap\operatorname{int}(\mathbb{D})$ exits~$\mathcal{C}_{\mathrm{BK}}(\Theta)$, violating convexity.
\end{proof}

\begin{remark}
The condition $\Theta_{11} \geq \Theta_{22}$ has a geometric interpretation in which the conic is at least as extended in the imaginary direction as in the real direction, i.e., tall conics. When $\Theta_{11} = \Theta_{22}$, the region reduces to a disk or half-plane.
\end{remark}

\subsection{Conic Containment via Frequency-Domain Certification} \label{ssec:conic}

We now establish a sufficient condition for SG containment in a conic region~$\mathcal{C}(\Theta)$.  For $H(s)\in\mathcal{RH}_\infty^{n\times n}$, we first define the \emph{Hermitian part} $H_s(\omega) := \tfrac{1}{2}(H(\textup{j}\omega)+H(\textup{j}\omega)^{*})$.

\begin{theorem}\label{thm:conic-containment}
Let $H(s)\in\mathcal{RH}_\infty^{n\times n}$, and let $\Theta\in\mathbb{S}^{3}$ be of the form~\eqref{eq:Theta-aligned}, indefinite, with $\Theta_{11}\geq\Theta_{22}$.  If $\forall  \omega\in \mathbb{R}$
\begin{align}\label{eq:Q-NSD}
  \begin{split}
      Q(\omega) := \alpha\, H_s(\omega)^{2}+ \Theta_{22}\, H(\textup{j}\omega)^{*}H(\textup{j}\omega) \\ +2\Theta_{13}\, H_s(\omega)+ \Theta_{33}\, I_n \preceq 0,
  \end{split}
\end{align}
then $\mathrm{SG}(H)\subseteq\mathcal{C}(\Theta)$.
\end{theorem}

\begin{proof}

\textit{Step~1 (SG points characterization).}
Fix $\omega\in\mathbb{R}$ and let $z\in\mathrm{SG}(H(\textup{j}\omega))$ be produced by input direction~$u\in\mathbb{C}^{n}$, $u\neq 0$.  By definition, $z = \rho\,e^{\pm\textup{j}\theta}$ where $\rho$ and $\theta$ are defined in \eqref{eqn:gain_phase}. From the polar form, $|z|^{2} = \rho^{2}$ and $r:=\operatorname{Re}\{z\} = \rho\cos\theta$.  Substituting the definitions of~$\rho$ and~$\theta$ into $\operatorname{Re}\{z\} = \rho\cos\theta$ yields $  \operatorname{Re}\{z\}= \tfrac{\|H(\textup{j}\omega)\,u\|}{\|u\|} \tfrac{\operatorname{Re}\{u^{*}H(\textup{j}\omega)\,u\}}{\|H(\textup{j}\omega)\,u\|\;\|u\|} = \tfrac{\operatorname{Re}\{u^{*}H(\textup{j}\omega)\,u\}}{\|u\|^{2}}.$
Since $\operatorname{Re}\{u^{*}A\,u\} = u^{*}\tfrac{1}{2}(A+A^{*})\,u$ for any~$A$, the two key quantities are
\begin{align}\label{eq:r-rho-RQ}
  r = \tfrac{u^{*}H_s(\omega)\,u}{\|u\|^{2}},\;  \rho^{2} = \tfrac{u^{*}H(\textup{j}\omega)^{*}H(\textup{j}\omega)\,u}{\|u\|^{2}}.
\end{align}

\textit{Step~2 (Conic membership).}
The conic membership test~\eqref{eq:conic-aligned} requires $v(z)^{\!\top}\Theta\,v(z)\leq 0$.  Writing $z = r + \textup{j}\,\operatorname{Im}\{z\}$ and using $\operatorname{Im}\{z\}^{2} = |z|^{2} - \operatorname{Re}\{z\}^{2} = \rho^{2} - r^{2}$, this becomes
\begin{align}\label{eq:conic-rho-r}
  (\Theta_{11} - \Theta_{22})\, r^{2} + \Theta_{22}\,\rho^{2} + 2\Theta_{13}\, r + \Theta_{33} \leq 0.
\end{align}

$r^2$ can be bounded above via the Cauchy--Schwarz inequality as follows $ (u^{*}H_s\,u)^{2} = |\langle H_s\,u, u\rangle|^{2} \leq \|H_s\,u\|^{2}\,\|u\|^{2} =(u^{*}H_s^{2}\,u)\,\|u\|^{2}$, where the last equality uses $H_s^{*} = H_s$.  Dividing by $\|u\|^{4}$ and multiplying by~$\alpha$,  and since $\Theta_{11} - \Theta_{22}=\alpha\geq 0$, we arrive to
\begin{align}\label{eq:CS-bound}
  \alpha\,r^{2}= \alpha\,\tfrac{(u^{*}H_s\,u)^{2}}{\|u\|^{4}}\leq \alpha\,\tfrac{u^{*}H_s^{2}\,u}{\|u\|^{2}}.
\end{align}

Replacing $\alpha\,r^{2}$ in~\eqref{eq:conic-rho-r} by the upper bound~\eqref{eq:CS-bound} and substituting~\eqref{eq:r-rho-RQ} for the remaining terms yields
\begin{align*}
  \alpha\, r^{2} + \Theta_{22}\,\rho^{2} + 2\Theta_{13}\, r + \Theta_{33}&\leq  \tfrac{u^{*}Q(\omega)\,u}{\|u\|^{2}},
\end{align*}
with $Q(\omega)$ as in \eqref{eq:Q-NSD}. Hence, since $Q(\omega)\preceq 0$, then~$\forall u\neq0$, ~\eqref{eq:conic-rho-r} holds for every $z\in\mathrm{SG}(H(\textup{j}\omega))$.  Since~$\omega$ is arbitrary, $\mathrm{SG}(H(\textup{j}\omega))\subseteq\mathcal{C}(\Theta)$ for all~$\omega\in\mathbb{R}$.

\textit{Step~3 (From frequency-wise to full SG containment).}
For an LTI system in~$\mathcal{RH}_\infty^{n \times n}$, the SG equals the h-convex hull of the frequency-wise SGs~\cite[Thm.~4]{Chaffey_2023}, i.e.,  $\mathrm{SG}(H) = \mathrm{co}_{h}\left(\cup_{\omega\in\mathbb{R}} \mathrm{SG}(H(\textup{j}\omega))\right)$. Since~$\Theta$ is indefinite with $\Theta_{11}\geq\Theta_{22}$, Theorem~\ref{thm:h-convexity} guarantees that $\mathcal{C}(\Theta)$ is h-convex.  Applying the Beltrami--Klein mapping~$f_{\mathrm{BK}}$ (See Appendix), h-convexity of~$\mathcal{C}(\Theta)$ in $\mathbb{C}$ becomes Euclidean convexity of~$f_{\mathrm{BK}}(\mathcal{C}(\Theta))$ in~$\mathbb{D}$.  By Step 2, $f_{\mathrm{BK}}(\mathrm{SG}(H(\textup{j}\omega))) \subset f_{\mathrm{BK}}(\mathcal{C}(\Theta))$ for every~$\omega$. Pick any two points $p \in f_{\mathrm{BK}}(\mathrm{SG}(H(\textup{j}\omega_1)))$ and $q \in f_{\mathrm{BK}}(\mathrm{SG}(H(\textup{j}\omega_2)))$ for arbitrary $\omega_1, \omega_2 \in \mathbb{R}$.  Since both $p$ and $q$ lie in the convex set $f_{\mathrm{BK}}(\mathcal{C}(\Theta))$, every convex combination $\lambda\, p + (1-\lambda)\,q$, $\lambda \in [0,1]$, also lies in $f_{\mathrm{BK}}(\mathcal{C}(\Theta))$.  As $p$, $q$, $\omega_1$, and $\omega_2$ are arbitrary, this shows that $\mathrm{co}_h\bigl(\bigcup_{\omega} f_{\mathrm{BK}}(\mathrm{SG}(H(\textup{j}\omega)))\bigr) \subseteq f_{\mathrm{BK}}(\mathcal{C}(\Theta))$.  Applying~$f_{\mathrm{BK}}^{-1}$, yields
$$ \mathrm{SG}(H) = \mathrm{co}_{h}\left(\bigcup_{\omega\in\mathbb{R}} \mathrm{SG}(H(\textup{j}\omega))\right) \subseteq \mathcal{C}(\Theta). $$
\end{proof}

\section{SRG-based Stability Analysis}

We now develop stability conditions that exploit the SG containment established in Section~\ref{sec:containment}. Let $H$ be a causal operator. The \emph{inverse SG} of $H$, denoted by $\SG^{-1}(H)$ is
\begin{align}
\SG^{-1}(H) \;:=\; \left\{ w \;\middle|\; w = z^{-1},\; z \in \SG(H)\setminus\{0\} \right\},
\label{eq:inverse_srg}
\end{align}
where the point $z=0$ is mapped to $w=\infty$. This construction represents the SG associated with the inverse input--output relation.
Moreover, let $\mathcal A, \mathcal B \subset \mathbb{C}$ be nonempty sets. The Euclidean distance between $\mathcal A$ and $\mathcal B$ is defined as
\begin{align}
\dist(\mathcal A,\mathcal B) \;:=\; \inf_{a \in \mathcal A,\, b \in \mathcal B} |a-b|. \label{eq:set_distance}
\end{align}
The sets $\mathcal A$ and $\mathcal B$ are said to be \emph{strictly separated} if $\dist(\mathcal A,\mathcal B) > 0$. Note that $\dist(\mathcal A,\mathcal B)=0$ if and only if the closures of $\mathcal A$ and $\mathcal B$ intersect or if $\infty\in\mathcal{A}\in\mathbb{C}\cup\{\infty\}$ and $\infty\in\mathcal{B}\in\mathbb{C}\cup\{\infty\}$.

\subsection{Hard and Soft Stability Theorems}

Feedback stability analysis within the SG framework relies on strict geometric separation between the SGs associated with the systems in feedback. We now recall the corresponding soft and hard SG stability theorems.

\begin{figure}[hbt] \centering \begin{tikzpicture}[scale=1, every node/.style={transform shape}] 
\draw (9.75,4.3) rectangle (11.25,3.7); \node at (10.5,4) {$H_1$}; 
\draw (9.75,3.6) rectangle (11.25,3); 
\node at (10.5,3.3) {$H_2$}; 
\draw[-latex, line width = .5 pt] (12.5,4) -- (13.5,4); 
\draw[-latex, line width = .5 pt] (11.25,4) -- (12.5,4) -- (12.5,3.3) -- (11.25,3.3);
\draw[fill = white] (8.5,4) circle [radius=0.05]; 
\draw[-latex, line width = .5 pt] (9.75,3.3) -- (8.5,3.3) -- (8.5,3.95); 
\draw[-latex, line width = .5 pt] (8.55,4) -- (9.75,4); 
\draw[-latex, line width = .5 pt] (7.5,4) -- (8.45,4); 
\node at (8.6892,3.8444) {\tiny$-$}; 
\node at (8.3332,4.1664) {\tiny$+$}; 
\node at (9,4.1664) {$e$}; 
\node at (8,4.1664) {$u$}; 
\node at (13,4.1664) {$y$}; \end{tikzpicture}
\caption{Negative feedback interconnection of $H_1$ and $H_2$.} 
\label{fig:feedback} 
\end{figure}
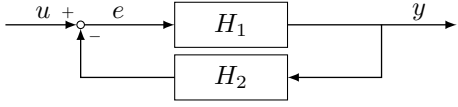

\begin{theorem}[Soft SG separation~\cite{Chaffey_2023, chen2025softhardSRG}]
\label{thm:soft_stability}
Let $H_1, H_2 : \mathcal{L}_2 \to \mathcal{L}_2$ be causal, $\mathcal{L}_2$-stable systems, and assume that their feedback interconnection as in Fig.~\ref{fig:feedback} is well-posed. If
\begin{align}
\dist\left(\SG^{-1}(H_1),\, -\tau\,\SG(H_2)\right) > 0, \qquad \forall \tau \in (0,1], \label{eq:soft_stability}
\end{align}
then the feedback interconnection is $\mathcal{L}_2$-stable.
\end{theorem}

The commonly imposed chord property~\cite{Chaffey_2023}, while sufficient for extracting explicit gain bounds from SG separation, is not required for feedback stability itself; stability follows solely from strict separation~\cite{chen2025softhardSRG}. 
Theorem~\ref{thm:soft_stability}  inherently excludes systems with integrators or marginally stable dynamics. Furthermore, the homotopy condition introduces a significant computational overhead. This motivates the hard SG framework, which simultaneously accommodates unbounded causal systems and simplifies the separation test to a single, static geometric check.

\begin{theorem}[Hard SG stability~\cite{chen2025softhardSRG}]
\label{thm:hard_stability}
Let $H_1, H_2 : \mathcal{L}_{2e} \to \mathcal{L}_{2e}$ be causal operators with a well-posed feedback interconnection as in Fig.~\ref{fig:feedback}. If
\begin{align}
\dist\left(\SGe^{-1}(H_1),\, -\SGe(H_2)\right) > 0, \label{eq:hard_distance}
\end{align}
then the feedback interconnection is $\mathcal{L}_2$-stable.
\end{theorem}

{We adopt the notion of well-posedness from}~\cite{Megretski1997}. Unlike the soft case, Theorem~\ref{thm:hard_stability} does not require $\mathcal{L}_2$ boundedness of the systems in feedback connection. The strict separation condition~\eqref{eq:hard_distance} enforces uniform geometric separation over all truncation horizons, which guarantees $\mathcal{L}_2$ stability of the closed-loop interconnection.

\subsection{Hard Stability Certification via Soft SG Regions}

For $\mathcal{L}_2$-stable systems, it is known that the soft and hard SGs satisfy the inclusion $\SG(H) \subseteq \overline{\SGe(H)}$~\cite{chen2025softhardSRG}, where $\overline{\SGe(H)}$ is the closure of the hard SG. In general, this relation is one-sided and does not allow hard stability certification from soft SG. However, when the associated multiplier admits is positive-negative, Theorem~\ref{thm:main} strengthens this relation by establishing equivalence between soft and hard SG containment within a multiplier-defined region. This observation enables hard stability certification using regions computed from soft SG.

\begin{corollary}[Hard stability via soft SG regions]
\label{cor:hard_stability}
Let $H_1, H_2 : \mathcal{L}_{2e} \to \mathcal{L}_{2e}$ be causal, $\Ltwo$-stable operators with a well-posed feedback interconnection as in Fig.~\ref{fig:feedback}. Assume there exist positive-negative multipliers $\Pi_1$ and $\Pi_2$ such that $\SG(H_i) \subset \mathcal{S}(\Pi_i), \, i \in \{1,2\}$. If
\begin{align}
\dist\left( \mathcal{S}^{-1}(\Pi_1), -\mathcal{S}(\Pi_2)\right) > 0,\label{eq:hard_separation}
\end{align}
then the feedback interconnection is $\mathcal{L}_2$-stable.
\end{corollary}

\begin{proof}
By Theorem~\ref{thm:main}, the soft containments $\SG(H_i) \subset \mathcal{S}(\Pi_i)$ imply the corresponding hard containments $\SGe(H_i) \subset \mathcal{S}(\Pi_i)$ for $i \in \{1,2\}$. Consequently,
\begin{align*}
    -\SGe(H_2) \subseteq -\mathcal{S}(\Pi_2), \qquad \SGe^{-1}(H_1) \subseteq \mathcal{S}^{-1}(\Pi_1).
\end{align*}
The separation condition~\eqref{eq:hard_separation} therefore yields
\begin{align*}
    \dist\!\big(-\SGe(H_2),\, \SGe^{-1}(H_1)\big)\! \geq\!\dist\!\big(-\mathcal{S}(\Pi_2), \mathcal{S}^{-1}(\Pi_1)\big)\!>\! 0,
\end{align*}
$\mathcal{L}_2$-stability follows from Theorem~\ref{thm:hard_stability}.
\end{proof}

 The practical benefit of Corollary~\ref{cor:hard_stability} is that one can compute regions for soft SGs and then apply Theorem~\ref{thm:hard_stability} for stability certification, avoiding both the $P \succeq 0$ constraint in hard SG computation and the homotopy sweep in soft stability tests. In general, the conservatism of regional containment depends on how well the chosen region approximates the true SG. To reduce it, one can certify containment in several regions intersect the results. When tight margin estimates are required, direct hard SG estimation methods~\cite{krebbekx2025computing,nauta2025computable} remain available.


\section{Numerical Examples}\label{sec:examples}
In this section, we illustrate the proposed SG-based stability analysis through representative numerical examples. 

\subsection{Circular Containment}
 
Consider the following $\mathcal{L}_2$-stable LTI systems
\begin{align}\label{eqn:sysproof}
H_1 = \begin{bmatrix} \tfrac{1}{(s+5)^2} & \tfrac{3}{(s+2)^2} \\ \tfrac{2}{s+10} & \tfrac{4}{(s+5)^2} \end{bmatrix}, \,
H_2 = \begin{bmatrix} \tfrac{1}{s+1} & \tfrac{0.3}{s+2} \\ -\tfrac{0.2}{s+3} & \tfrac{1}{s+1} \end{bmatrix}.
\end{align}
For each system, we compute both the soft and hard SGs, together with circular multiplier regions $\mathcal{S}(\Pi_1)$ and $\mathcal{S}(\Pi_2)$. Specifically, the region $\mathcal{S}(\Pi_1)$ has center $c = 0.1$ and radius $r = 0.78$ and is shown in Fig.~\ref{fig:example_H1}, while $\mathcal{S}(\Pi_2)$ has center $c = 0.52$ and radius $r = 0.75$ and is depicted in Fig.~\ref{fig:example_H2}. Both multiplier regions satisfy the conditions of Proposition~\ref{prop:circular}. Figure~\ref{fig:examples} illustrates the geometric containments
\begin{align*}
    \SG(H_i) \subset \mathcal{S}(\Pi_i), \quad\SGe(H_i) \subset \mathcal{S}(\Pi_i),
\end{align*}
which are consistent with the theoretical predictions of Theorem~\ref{thm:main}. Notably, the hard SG occupies a larger region of the complex plane than its soft counterpart, yet both are strictly contained within the predicted multiplier region $\mathcal{S}(\Pi_i)$. The soft SGs are obtained via frequency-domain sampling\cite{krebbekxGraphicalAnalysisNonlinear2025}, whereas the hard SGs are computed using the numerical algorithm proposed in~\cite{krebbekx2025computing}. Closed-loop stability of the feedback interconnection follows from the strict separation condition \eqref{eq:hard_separation} which is observed in Fig.~\ref{fig:example_H3}. The stability margin, $\vartheta_{\Pi_{1,2}}\! =\! \dist\!\left(\mathcal{S}^{-1}(\Pi_1),-\mathcal{S}(\Pi_2)\right)$, depends on the particular choice of multipliers and is, in general, smaller than the margin obtained from the exact hard SGs. In the present example, the margin can be computed directly from the endpoints of the circular regions, yielding $\vartheta_{\Pi_{1,2}} = 0.2$.

\begin{figure}[ht]
    \centering
\begin{subfigure}{.5\linewidth}
     \centering 
    \includegraphics[width=1\linewidth]{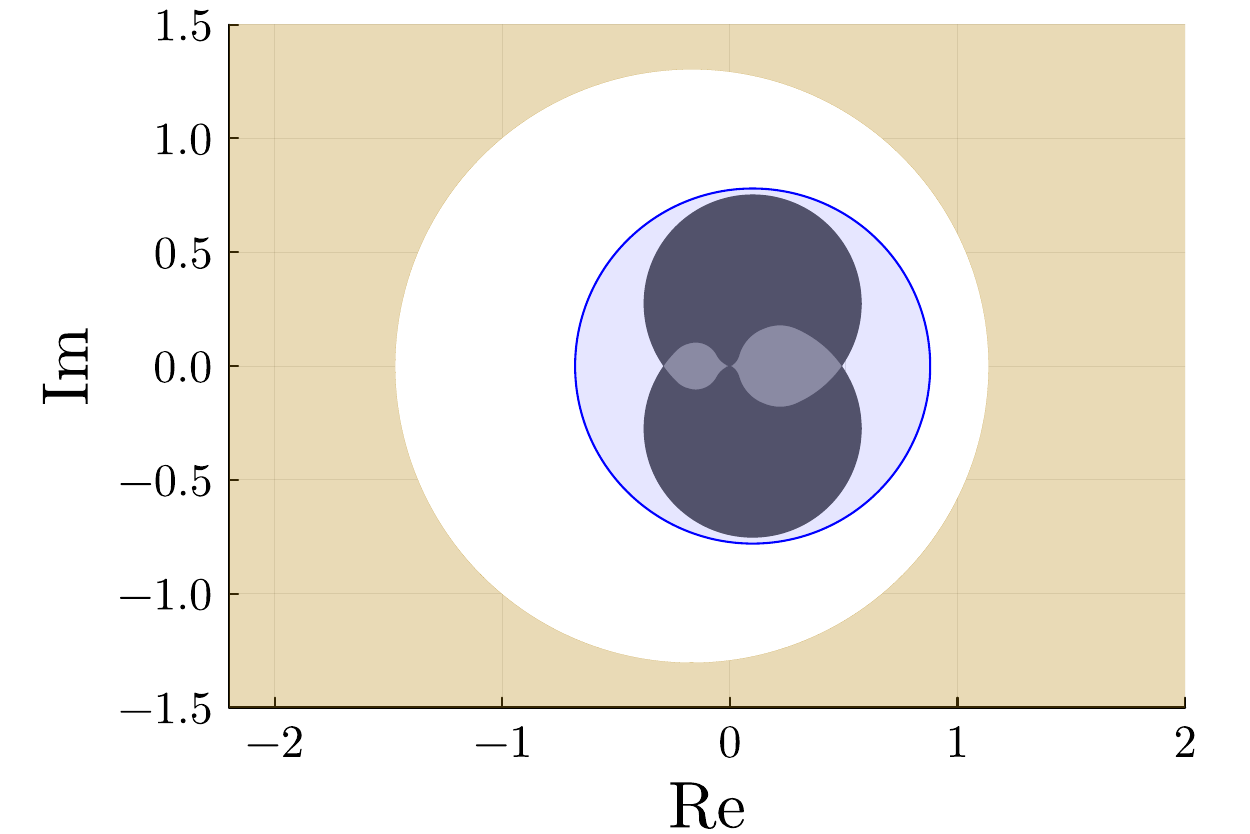}
    \caption{}
    \label{fig:example_H1}
\end{subfigure}%
\begin{subfigure}{.5\linewidth}
     \centering 
    \includegraphics[width=1\linewidth]{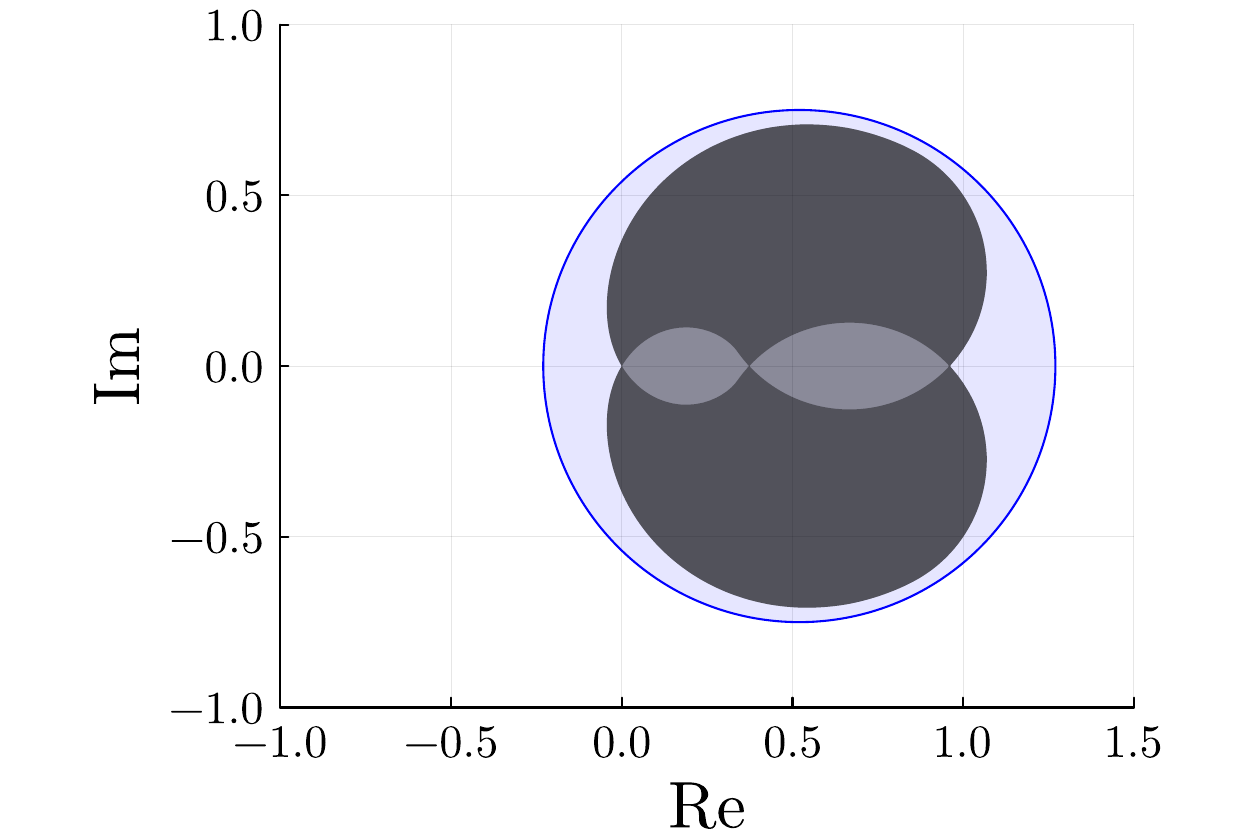}
    \caption{}
    \label{fig:example_H2}
\end{subfigure}%

\begin{subfigure}{.6\linewidth}
     \centering 
    \includegraphics[width=1\linewidth]{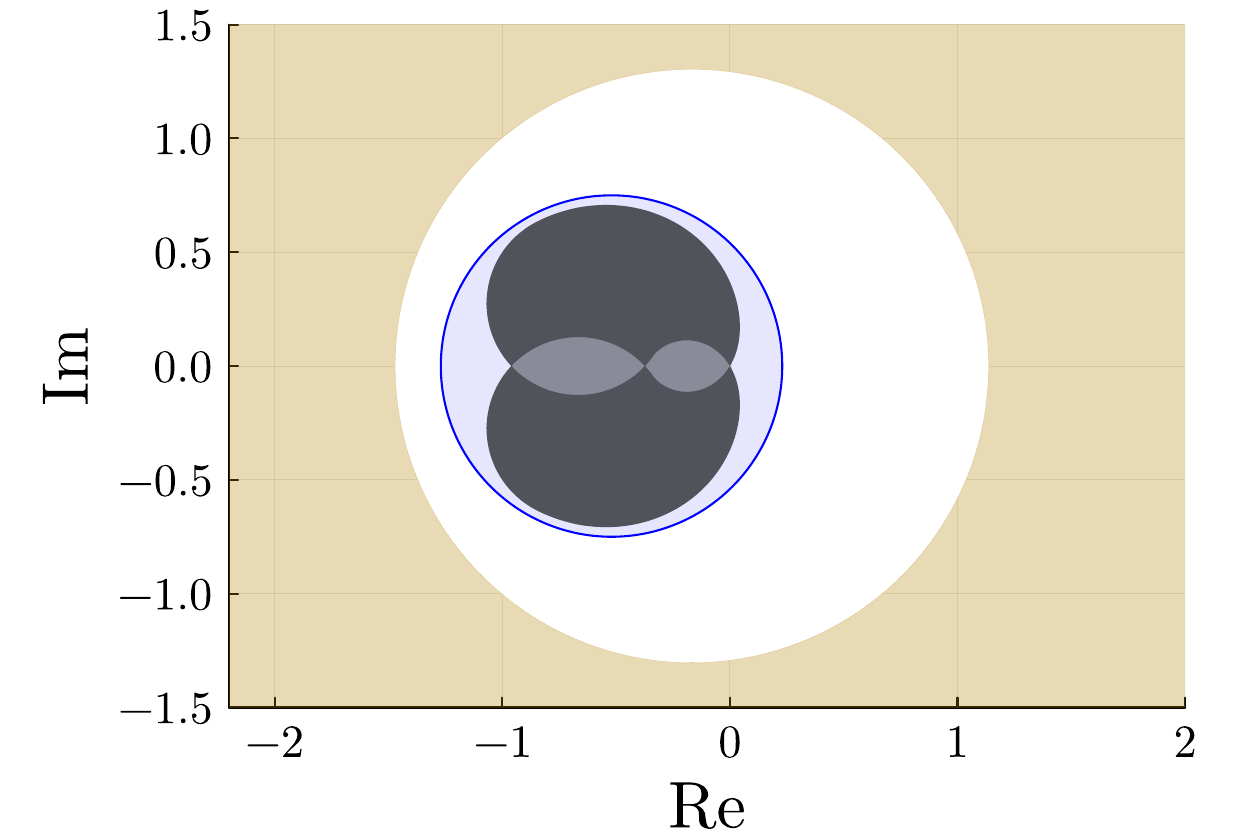}
    \caption{}
    \label{fig:example_H3}
\end{subfigure}%
\caption{
(a) For system $H_1$, soft SG (dark gray) and hard SG (light gray), together with the multiplier region $\mathcal{S}(\Pi_1)$ (blue) and its inverse $\mathcal{S}^{-1}(\Pi_1)$ (yellow). 
(b) For system $H_2$, soft SG (dark gray) and hard SG (light gray), along with the multiplier region $\mathcal{S}(\Pi_2)$ (blue). 
(c) Corresponding sets used for stability analysis: $-\SG(H_2)$ (dark gray), $-\SGe(H_2)$ (light gray), $-\mathcal{S}(\Pi_2)$ (blue), and $\mathcal{S}^{-1}(\Pi_1)$ (yellow).}
    \label{fig:examples}
\end{figure} 

\subsection{Conic Containment}
\label{sec:conic_example}
To illustrate the advantage of conic over circular containment, consider $H_1$ and $H_2$ in \eqref{eqn:sysproof}. Figure~\ref{fig:conic_example} compares the ellipsoidal region with its smallest enclosing disk. The disk must match the maximum imaginary excursion, leading to excess coverage along the real axis where the SG is narrow. The ellipse exploits this nonuniformity, reducing the area by approximately $21\%$ for $H_1$. For $H_2$, the reduction is more modest, around $9\%$ relative to the circular multiplier. 
 
\begin{figure}[ht]
\centering
\begin{subfigure}{.5\linewidth}
     \centering 
    \includegraphics[width=1\linewidth]{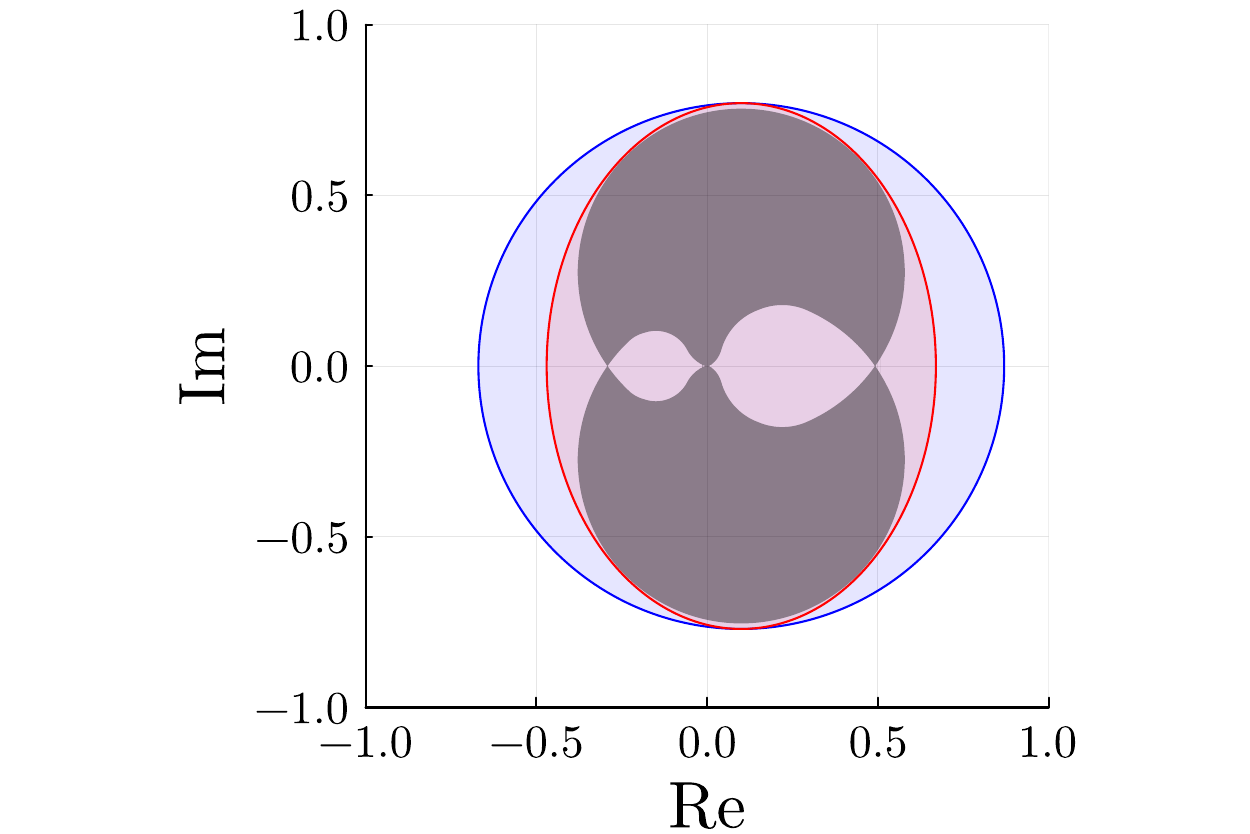}
    \caption{}
\end{subfigure}%
\begin{subfigure}{.5\linewidth}
     \centering 
    \includegraphics[width=1\linewidth]{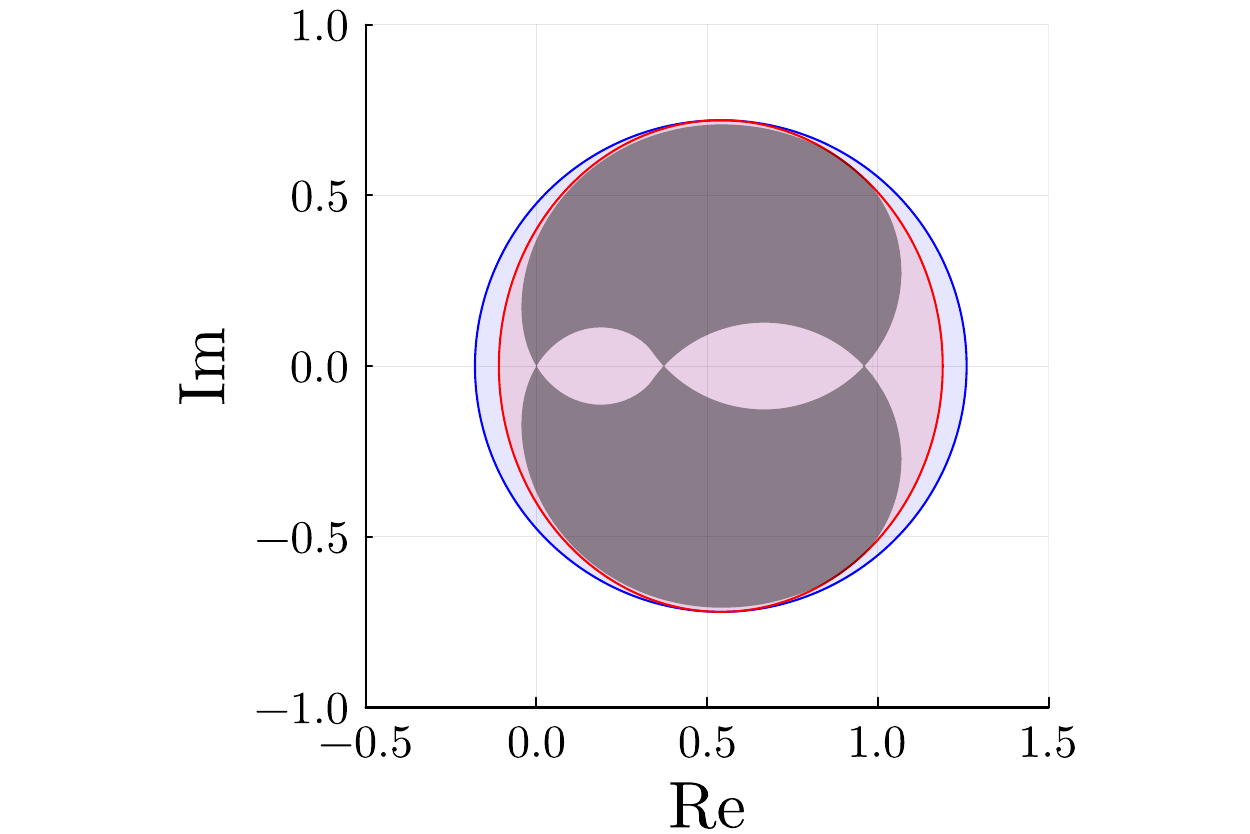}
    \caption{}
\end{subfigure}%
\caption{(a) System $H_1$: soft SG (gray), circular region $\mathcal{S}(\Pi_1)$ (blue), and ellipsoidal region $\mathcal{C}(\Theta_1)$ (red).
(b) System $H_2$: soft SG (gray), circular region $\mathcal{S}(\Pi_2)$ (blue), and ellipsoidal region $\mathcal{C}(\Theta_2)$ (red).}
\label{fig:conic_example}
\end{figure}

\subsection{Scalability Demonstration}
\label{sec:scalability}

To quantify the computational advantages of Corollary~\ref{cor:lmi}, we compare the soft and hard LMI formulations across systems of increasing dimension. We consider block-diagonal MIMO systems constructed from first-order subsystems
\begin{align*}
    H_k(s) = \tfrac{1}{s + a_k}, \qquad a_k \in [0.1, 0.3].
\end{align*}
Block-diagonal structures model systems where subsystems are intrinsically decoupled and interact only through a shared interconnection.
Therefore, the selected system reflects physical modularity of several systems~\cite{Baron2025decentralized,baronprada2026powersystems}. We consider system dimensions $m \in \{10,25, 50, 75, 100, 125, 150, 200, 250, 300\}$, and, solve both the soft LMI~\eqref{eq:soft_lmi} and the hard LMI augmented with the constraint $P \succeq 0$, using interior disk multipliers ($\Pi_{\textrm{int}}$) satisfying the positive--negative property. All experiments are conducted using SeDuMi~\cite{sturm1999using} and MOSEK~\cite{MOSEKbook}.

Figure~\ref{fig:scalability} summarizes the {computational performance} of the proposed approach. The soft formulation outperforms the hard formulation, with speedups defined as the ratio between the runtime of the hard and soft LMIs. These speedups are not monotonic in problem size, ranging from  $1.18\times$ at $m=10$ to $1.79\times$ at $m=75$ using SeDuMi, corresponding to runtime reductions of roughly 15--44\%. The non-monotonic nature of the speedup factor likely stems from the varying efficiency of the solvers internal heuristics.
Since runtimes increase with state dimension, these reductions can translate into minutes to hours of computational savings. These results show that Theorem~\ref{thm:main} exploits the $J$-spectral factorization allowing soft SG computation while certifying hard containment, yielding significant computational savings for large-scale systems.
 
\begin{figure}[ht]
\centering
\includegraphics[width=1\columnwidth]{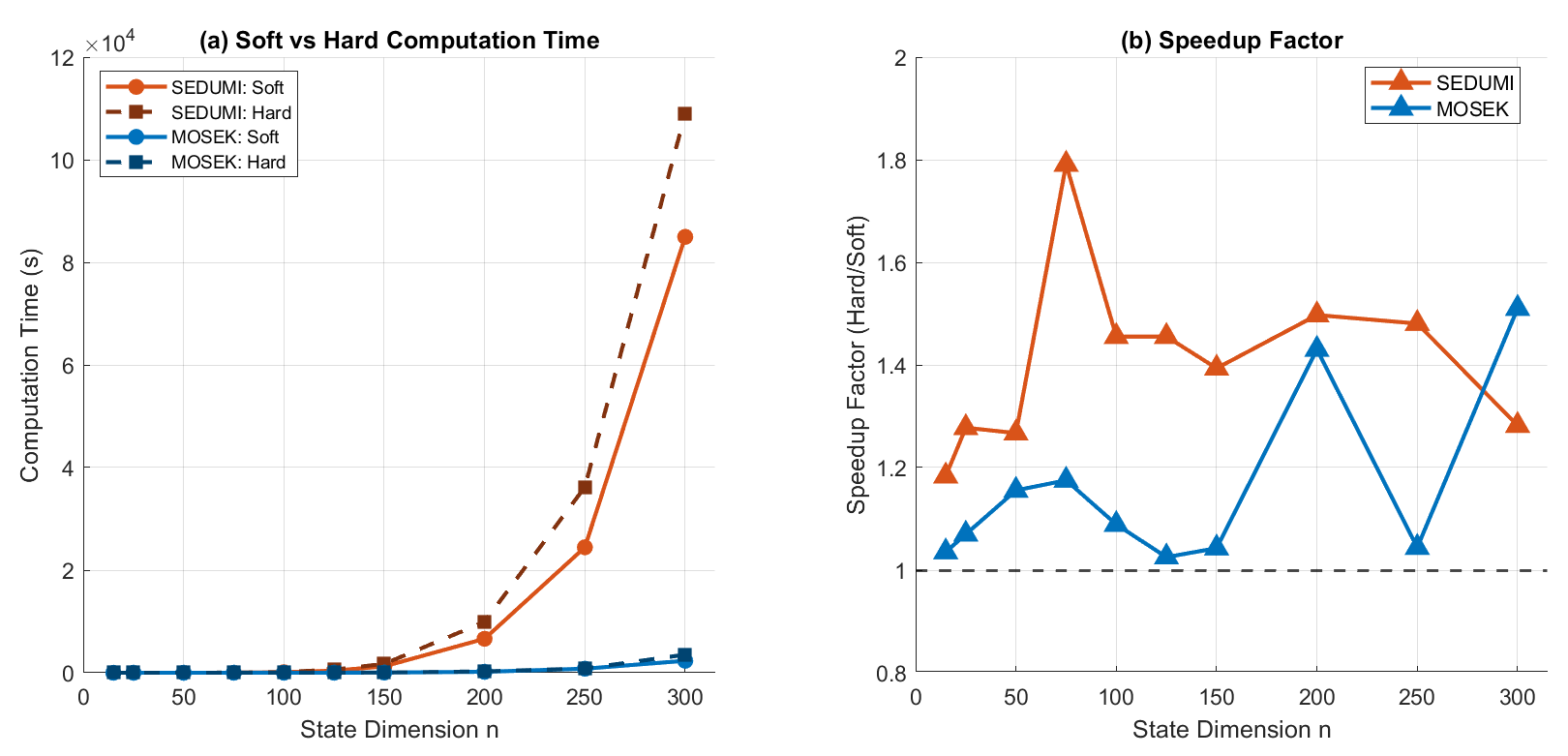}
\caption{Scalability results for systems with state dimension 
$m \in \{10, 25, \ldots, 300\}$. (a) Computation time for soft 
($P$ free) and hard ($P \succeq 0$) formulations. (b) Speedup factor. }
\label{fig:scalability}
\end{figure}

\section{Conclusion}

This paper establishes that soft and hard SG containment coincide for positive-negative multipliers, and extends the SG containment framework from circular to conic regions through an h-convexity characterization and a frequency-domain certification condition.  Together, these results represent a further step for SG-based stability certification practical for large-scale LTI systems by eliminating both the $P \succeq 0$ constraint and the homotopy sweep, while opening the door to non-circular exclusion specifications.
 
Several directions remain open.  First, extending the conic containment results to the hard SG setting would complete the parallel between the circular and conic frameworks.  Second, applying the framework to black-box or data-driven system representations, where transfer function models are unavailable, is a natural next step for practical deployment.

\bibliographystyle{IEEEtran}
\bibliography{bibtex/TIE}

@article{pates2021scaled,
  title={The scaled relative graph of a linear operator},
  author={Pates, Richard},
  journal={arXiv preprint arXiv:2106.05650},
  year={2021}
}

@article{Chaffey_2023,
title={Graphical Nonlinear System Analysis},
volume={68},
ISSN={2334-3303},
DOI={10.1109/tac.2023.3234016},
number={10},
journal={IEEE Transactions on Automatic Control},
publisher={Institute of Electrical and Electronics Engineers (IEEE)},
author={Chaffey, Thomas and Forni, Fulvio and Sepulchre, Rodolphe},
year={2023},
month=oct, pages={6067–6081} }

@book{ryu2022large,
title={Large-Scale Convex Optimization: Algorithms \& Analyses via Monotone Operators},
author={Ryu, E.K. and Yin, W.},
isbn={9781009191067},
year={2022},
publisher={Cambridge University Press}
}

@book{zhou1998,
  title={Essentials of robust control},
  author={Zhou, Kemin and Doyle, John Comstock},
  volume={104},
  year={1998},
  publisher={Prentice hall Upper Saddle River, NJ}
}

@article{Baron2025SRG,
  title={Stability results for {MIMO} {LTI} systems via Scaled Relative Graphs},
  author={Baron-Prada, Eder and Padoan, Alberto and Anta, Adolfo  and Dörfler, Florian},
  journal={arXiv preprint arXiv:2503.13583},
  year={2025}
}

@article{Baron2025decentralized,
  author={Baron-Prada, Eder and Anta, Adolfo and Dörfler, Florian},
  journal={IEEE Control Systems Letters}, 
  title={On Decentralized Stability Conditions Using Scaled Relative Graphs}, 
  year={2025},
  volume={9},
  number={},
  pages={691-696},}

@article{chen2025softhardSRG,
      title={Soft and Hard Scaled Relative Graphs for Nonlinear Feedback Stability}, 
      author={Chao Chen and Sei Zhen Khong and Rodolphe Sepulchre},
      year={2025},
      journal={ArXiv:2504.14407},
}

@article{Carrasco_2018,
   title={Conditions for the equivalence between {IQC} and graph separation stability results},
   volume={92},
   ISSN={1366-5820},
   number={12},
   journal={International Journal of Control},
   publisher={Informa UK Limited},
   author={Carrasco, Joaquin and Seiler, Peter},
   year={2018},
   month=apr, pages={2899–2906} }

@article{krebbekx2025computing,
      title={Computing the Hard Scaled Relative Graph of {LTI} Systems}, 
      author={Julius P. J. Krebbekx and Eder Baron-Prada and Roland Tóth and Amritam Das},
      year={2025},
      journal={ArXiv:2511.17297},
}

@article{degroot2025dissipativity,
      title={A Dissipativity Framework for Constructing Scaled Graphs}, 
      author={Timo de Groot and Maurice Heemels and Sebastiaan van den Eijnden},
      year={2025}, 
      journal={ArXiv:2507.08411},
}

@INPROCEEDINGS{Sebastian2025_signed,
  author={van den Eijnden, Sebastiaan and Chen, Chao and Scheres, Koen and Chaffey, Thomas and Lanzon, Alexander},
  booktitle={2025 IEEE 64th Conference on Decision and Control (CDC)}, 
  title={On phase in scaled graphs}, 
  year={2025},
  volume={},
  number={},
  pages={3595-3600}}

@INPROCEEDINGS{Carrasco_2015_IQCsseparation,
  author={Carrasco, Joaquin and Seiler, Peter},
  booktitle={2015 54th IEEE Conference on Decision and Control (CDC)}, 
  title={Integral quadratic constraint theorem: A topological separation approach}, 
  year={2015},
  volume={},
  number={},
  pages={5701-5706},
  doi={10.1109/CDC.2015.7403114}}

@ARTICLE{Megretski1997,
  author={Megretski, A. and Rantzer, A.},
  journal={IEEE Transactions on Automatic Control}, 
  title={System analysis via integral quadratic constraints}, 
  year={1997},
  volume={42},
  number={6},
  pages={819-830},
  doi={10.1109/9.587335}}

@article{nauta2025computable,
      title={Computable Characterisations of Scaled Relative Graphs of Closed Operators}, 
      author={Talitha Nauta and Richard Pates},
      year={2025},
      eprint={2511.08420},
      journal={arXiv:2511.08420}
}

@article{baronprada2026powersystems,
  author={Baron-Prada, Eder and Anta, Adolfo and Dörfler, Florian},
  journal={IEEE Transactions on Power Systems}, 
  title={Stability Analysis of Power-Electronics-Dominated Grids Using Scaled Relative Graphs}, 
  year={2026},
  volume={},
  number={},
  pages={1-15},
  doi={10.1109/TPWRS.2026.3674752}}

@article{krebbekxGraphicalAnalysisNonlinear2025,
  title={Graphical Analysis of Nonlinear Multivariable Feedback Systems},
  author={Krebbekx, Julius P. J. and T{\'o}th, Roland and Das, Amritam},
  note={Submitted to IEEE-TAC},
  year={2025},
  journal={arXiv:2507:16513}
}

@book{nesterov1994interior,
  title={Interior-point polynomial algorithms in convex programming},
  author={Nesterov, Yurii and Nemirovskii, Arkadii},
  year={1994},
  publisher={SIAM}
}

@book{MOSEKbook,
  title={The MOSEK Optimization Toolbox for MATLAB Manual},
  author={{MOSEK ApS}},
  year={2023},
  note={Version 10.1}
}

@article{sturm1999using,
  title={Using SeDuMi 1.02, a MATLAB toolbox for optimization over symmetric cones},
  author={Sturm, Jos F},
  journal={Optimization methods and software},
  year={1999},
  publisher={Taylor \& Francis}
}

@ARTICLE{Seiler2015,
  author={Seiler, Peter},
  journal={IEEE Transactions on Automatic Control}, 
  title={Stability Analysis With Dissipation Inequalities and Integral Quadratic Constraints}, 
  year={2015},
  volume={60},
  number={6},
  pages={1704-1709}}

@article{IwasakiHara2005,
  author    = {T. Iwasaki and S. Hara},
  title     = {Generalized {KYP} Lemma: Unified Frequency Domain
               Inequalities with Design Applications},
  journal   = {IEEE Transactions on Automatic Control},
  volume    = {50},
  number    = {1},
  pages     = {41--59},
  year      = {2005}
}

@article{GOLDMAN2005632,
title = {Curvature formulas for implicit curves and surfaces},
journal = {Computer Aided Geometric Design},
volume = {22},
number = {7},
pages = {632-658},
year = {2005},
note = {Geometric Modelling and Differential Geometry},
issn = {0167-8396},
author = {Ron Goldman},
}

@article{Gupta_2023_elipsoidal,
title = {Data-driven IQC-Based Uncertainty Modelling for Robust Control Design},
journal = {IFAC-PapersOnLine},
volume = {56},
number = {2},
pages = {4789-4795},
year = {2023},
note = {22nd IFAC World Congress},
issn = {2405-8963},
doi = {https://doi.org/10.1016/j.ifacol.2023.10.1244},
author = {Vaibhav Gupta and Elias Klauser and Alireza Karimi},
}

@book{do2016differential,
  title={Differential geometry of curves and surfaces: revised and updated second edition},
  author={Do Carmo, Manfredo P},
  year={2016},
  publisher={Courier Dover Publications}
}

\end{document}